\documentclass[10pt]{amsart}
\usepackage[top=2.9cm,bottom=2.8cm,left=4cm,right=4cm]{geometry}               
\usepackage{amsmath,amscd,amssymb,amsthm}
\usepackage[english]{babel}
\usepackage{mathtools}
\usepackage{enumerate}
\usepackage{setspace}
\usepackage{mathrsfs}
\usepackage{xcolor}
\usepackage[numbers]{natbib}
\usepackage[hidelinks]{hyperref}

\usepackage{subfig}
\usepackage{multirow} 
\usepackage{color}
\usepackage{tikz}
\usepackage{float}
\usepackage{mdwlist}

\DeclareMathOperator{\bigO}{O}
\DeclareMathOperator{\PO}{PO}

\DeclareMathOperator{\SbigO}{SO}

\DeclareMathOperator{\Mat}{Mat}
\DeclareMathOperator{\GL}{GL}
\DeclareMathOperator{\PSL}{PSL}

\DeclareMathOperator{\Tr}{Tr}
\DeclareMathOperator{\Ad}{Ad}
\DeclareMathOperator{\Isom}{Isom}
\DeclareMathOperator{\arcosh}{arcosh}

\DeclareMathOperator{\Span}{Span}
\DeclareMathOperator{\diag}{diag}
\DeclareMathOperator{\sgn}{sgn}

\DeclareMathOperator{\Ram}{Ram}
\DeclareMathOperator{\Br}{Br}

\DeclareMathOperator{\radic}{rad}

\newtheorem{thm}{Theorem}[section]

\newtheorem{lemma}[thm]{Lemma}
\newtheorem{cor}[thm]{Corollary}
\newtheorem{prop}[thm]{Proposition}

\theoremstyle{definition}
\newtheorem{defn}[thm]{Definition}
\newtheorem{rmk}[thm]{Remark}
\newtheorem{caution}[thm]{Caution}

\newtheorem{ex}[thm]{Example}

\newtheorem*{theorem*}{Theorem}

\def\gm{\hbox{$\bullet\kern-4pt $ --\kern-3pt\raise
6pt\hbox{$\scriptstyle m$}\kern-8pt --\kern-6pt ---\kern-5pt
--\kern-1pt$\bullet$}}

\makeindex

 \author[E. Dotti]{Edoardo Dotti}
 \address{Department of Mathematics\\ 
 University of Fribourg\\
 1700 Fribourg, Switzerland\\
 }
 \email{edoardo.dotti@unifr.ch}

\date{}
\title{On the commensurability of hyperbolic Coxeter groups}
\begin{document}
\begin{abstract}
In this paper we study the commensurability of hyperbolic Coxeter groups of finite covolume, providing three necessary conditions for commensurability. Moreover we tackle different topics around the field of definition of a hyperbolic Coxeter group: which possible fields can arise, how this field determines a range of possible dihedral angles of a Coxeter polyhedron and we provide two new sets of generators for quasi-arithmetic groups. This work is a concise version of chapters 4 and 5 of the author's Ph.D. thesis \citep{dotti2020}.
\end{abstract}
\maketitle
\section{Introduction}
For $n\geq 2$, let $\mathbb{H}^n$ be the real hyperbolic space of dimension $n$ and denote by $\Isom(\mathbb{H}^n)$ its isometry group. Consider a space form $\mathbb{H}^n/\Gamma$, where $\Gamma$ is a discrete subgroup of $\Isom(\mathbb{H}^n)$. Two such space forms are commensurable if they admit a common finite-sheeted cover. We are interested in distinguishing hyperbolic space forms up to commensurability. \\

The situation is well understood in dimensions two and three. For $n=3$ the group $\Isom^+(\mathbb{H}^3)$ of orientation preserving isometries can be identified with the group $\PSL(2,\mathbb{C})$. Due to the work of Maclachlan and Reid \citep{maclachlan2013arithmetic} there are two powerful commensurability invariants for Kleinian groups in $\PSL(2,\mathbb{C})$, the \textit{invariant trace field} and \textit{invariant quaternion algebra}, which form a complete set of invariants for arithmetic Kleinian groups.

In higher dimensions the situation needs more investigation. When dealing with arithmetic (of the simplest type) hyperbolic lattices, Gromov and Piatetski-Shapiro \citep{gromov1987non} provide a complete commensurability criterion. Consider an arithmetic lattice with its associated totally real field and quadratic form. Then, their commensurability criterion states that two such lattices are commensurable if and only if the two associated fields coincide and the two forms are similar over their field.

However, in the non-arithmetic context, no general commensurability criterion is known up to date.\\

In this paper we study the problem of commensurability of hyperbolic Coxeter groups of finite covolume. These are discrete subgroups of $\Isom(\mathbb{H}^n)$ generated by finitely many reflections in the bounding hyperplanes of Coxeter polyhedra, which are polyhedra whose angles are integral submultiples of $\pi$. We always suppose that the volume of the Coxeter polyhedra are finite. Following the work of Vinberg in \citep{vinberg1967discrete}, we shall associate a field of cycles and a quadratic form to every hyperbolic Coxeter group: the \textit{Vinberg field} and the \textit{Vinberg form}. Inspired by the result of Gromov and Piatetski-Shapiro, we prove the following.

\begin{theorem*}
Let $\Gamma_1$ and $\Gamma_2$ be two commensurable cofinite hyperbolic Coxeter groups acting on $\mathbb{H}^n$, $n\geq 2$. Then their Vinberg fields coincide and the two associated Vinberg forms are similar over this field.
\end{theorem*}

We are then able to refine the previous theorem by associating to a Coxeter group as above a ring, the \textit{Vinberg ring}, and show that this ring is also a commensurability invariant.\\

After that we focus on which field can arise as the Vinberg field of a quasi-arithmetic hyperbolic Coxeter group and then we study how the extension degree of the Vinberg field effects the possible angles of its Coxeter polyhedron. The paper concludes with two new sets of generators for the Vinberg field of a quasi-arithmetic Coxeter group. Specifically, we shall see that the Vinberg field of such a Coxeter group is generated by the coefficients of the characteristic polynomial of its Gram matrix on one side and by the coefficients of the characteristic polynomial of any Coxeter transformation on the other side. \\

The paper is structured as follows. In Section \ref{prelim} we present all the theory needed for the rest of the paper such as hyperbolic Coxeter groups and commensurability. In Section \ref{commens} we prove the Theorem above, the commensurability property of the Vinberg ring and we discuss the similarity classification of Vinberg forms. In the last section we study the Vinberg field and provide new sets of generators as mentioned above. All the results will be supported by examples.\\
This work is a concise version of chapters 4 and 5 of the author's Ph.D. thesis \citep{dotti2020}, and the proofs in this work are a direct adaptation from those in \citep{dotti2020}.

\subsection*{Acknowledgments} 
The author would like to thank Prof. Dr. Ruth Kellerhals for the support, valuable discussions and helpful suggestions. This work was supported by the Swiss National Science Foundation, projects 200020\_156104 and 200021\_172583.

\section{Preliminaries}\label{prelim}

\subsection{Hyperbolic space, Coxeter polyhedra and Coxeter groups}\label{threegeom}

Let $n\geq 2$ and denote by $\mathbb{H}^n$ the real hyperbolic space of dimension $n$. We use here the \textit{vector space model}, or \textit{hyperboloid model}. Equip $\mathbb{R}^{n+1}$ with the \textit{Lorentzian product} defined as
\[
\langle x, y\rangle=\sum_{i=1}^{n}x_i y_i- x_{n+1}y_{n+1}.
\]
The hyperboloid model for $\mathbb{H}^n$ is then given by the set
\[
\mathcal{H}^n:=\{ x\in\mathbb{R}^{n+1}\mid\left| \left| x\right| \right|^2= \langle x, x\rangle=-1, \text{ } x_{n+1}>0 \}
\]
with metric $d_{\mathcal{H}^n}(x,y)=\arcosh(-\langle x, y\rangle)$ for all $x,y\in\mathcal{H}^n.$

The space $\mathbb{R}^{n+1}$ equipped with the Lorentzian product is denoted by $\mathbb{R}^{n,1}$. Its associated quadratic form, the \textit{Lorentzian form}, will be denoted by $q_{-1}(x):=\langle x, x\rangle$.

The group of isometries $\Isom(\mathcal{H}^n)$ is the Lie group of \textit{positive Lorentzian matrices} 
\begin{equation}\label{eq:lorentzmatrices}
\bigO^+(n,1)=\left\lbrace A\in\Mat(n+1,\mathbb{R})\mid A^{T}JA=J,\text{ } [A]_{n+1,n+1}>0 \right\rbrace,
\end{equation}
where $J=\diag(1,\dots, 1,-1)$ is the diagonal matrix which represents the Lorentzian form.
\begin{rmk}
The group $\bigO^+(n,1)$ is \textit{not} an algebraic group. However, one can project $\mathcal{H}^n$ to the open unit ball and consider \textit{its projective model} $\mathcal K^n$. A very important aspect of $\mathcal{K}^n$ is that its isometries form an algebraic group. Consider the group of all matrices which preserve the Lorentzian form $\left\lbrace A\in\Mat(n+1,\mathbb{R})\mid A^{T}JA=J\right\rbrace =\bigO(n,1)$. Then,
\begin{equation}\label{equationprojectiveiso}
\Isom(\mathcal{K}^n)\cong \bigO(n,1)/\{\pm I\}=:\PO(n,1).
\end{equation}
The fact that $\Isom(\mathcal{K}^n)$ is an algebraic group will be exploited in Section \ref{Vinconstruction} (see Remark \ref{rmk:Choi}).
\end{rmk}

Each hyperbolic hyperplane $H_e=e^{\perp}$ is given as the orthogonal complement of a vector $e\in\mathbb{R}^{n+1}$ of Lorentzian norm $1$, that is, $
H_e=\{x\in\mathbb{R}^{n+1}\mid\langle x, e\rangle=0\}$.

A hyperplane $H_e$ divides $\mathcal{H}^n$ into two half-spaces $H_e^-=\{x\in\mathcal{H}^n\mid \langle x, e\rangle\leq 0\}$ and $H_e^+=\{x\in\mathcal{H}^n\mid \langle x, e\rangle\geq 0\}$ such that $H_e^-\cap H_e^+=H_e$. A (convex) \textit{polyhedron} $P\subset\mathcal{H}^n$ is the intersection with non-empty interior of finitely many half-spaces, that is, 
\[
P=\bigcap_{i=1}^N H_{e_i}^-,
\]
$N\geq n+1$, where the unit vector $e_i$ normal to the hyperplane $H_{e_i}$ is pointing outwards of $P$. If $N=n+1$, then $P$ is called an \textit{$n$-simplex}.

Particularly, a \textit{Coxeter polyhedron} is a polyhedron all of whose angles between its bounding hyperplanes are either zero or sub-multiples of $\pi$, hence of the form $\frac{\pi}{k}$ for $k\in\mathbb{N}$, $k\geq 2$.\\

Consider a hyperplane $H_e$ in $\mathcal{H}^n$. A \textit{reflection} with respect to the hyperplane $H_e$ is the application $s_e=s_{H_e}:\mathcal{H}^n\rightarrow\mathcal{H}^n$ defined as $
s_e(x)=x-2\langle x, e\rangle\, e$ and satisfying $s_e^2=1$.

Let $\Gamma=\langle s_{e_1},\dots,s_{e_N}\rangle<\Isom(\mathcal{H}^n)$ be the discrete group generated by the reflections in the hyperplanes bounding a Coxeter polyhedron $P=\bigcap_{i=1}^N H_{e_i}^{-}$ in $\mathcal{H}^n$. If two hyperplanes $H_{e_i}$ and $H_{e_j}$ in the boundary of $P$ intersect under an angle $\pi/m_{ij}$, $m_{ij}\geq 2$, then $(s_{e_i}s_{e_j})^{m_{ij}}=1$ in $\Gamma$. If $H_{e_i}$ and $H_{e_j}$ are parallel or ultraparallel, $s_{e_i}s_{e_j}$ is of infinite order in $\Gamma$. In this way, $\Gamma$ represents an abstract Coxeter group. The group $\Gamma$ is called a \textit{hyperbolic Coxeter group}. The number of its generating reflections $N$ is called the rank of $\Gamma$.

In the sequel hyperbolic Coxeter groups will always be assumed to be \textit{cofinite}, that is, the associated Coxeter polyhedron $P$ has finite volume. A hyperbolic Coxeter group $\Gamma$ is said to be \textit{cocompact} if $P$ is compact.

A hyperbolic Coxeter group and its Coxeter polyhedron can be most conveniently described by means of its Gram matrix and its Coxeter graph as follows.

For $\Gamma<\Isom(\mathcal{H}^n)$ a Coxeter group of rank $N$ with Coxeter polyhedron $P=\bigcap_{i=1}^N H_{e_i}^-$, $N\geq n+1$, the \textit{Gram matrix} associated to $P$ and to $\Gamma$ is the real symmetric matrix $G:=G(P)=G(\Gamma)=(g_{ij})_{1\leq i,j\leq N}$ with coefficients 
\[
g_{ij}=\langle e_i, e_j\rangle.
\]

The Gram matrix $G$ of a Coxeter group $\Gamma$ is unique up to enumeration of the hyperplanes and has signature $(n,1)$ (see \citep[Chapter 6]{vinberg1988geometry}). Moreover, a \textit{cycle} (or \textit{cyclic product}) of $G$ is defined as
\[
g_{i_1 i_2}g_{i_2 i_3}\dots g_{i_{l-1} i_{l}}g_{i_{l} i_{1}}
\]
for any $\{i_1, i_2,\dots ,i_l\}\subset\{1,2,\dots, m\}$.
A cycle is called \textit{simple} if the indices $i_j$ in the cycle are all distinct.

The \textit{Coxeter graph} of $\Gamma$ is the graph with $N$ vertices for which the vertex $i$ corresponds to the hyperplane $H_{e_i}$. Between two vertices $i$ and $j$ we have:

\begin{enumerate}[i)]
\item an edge if the angle between $H_{e_i}$ and $H_{e_j}$ is $\pi/k$, $k\geq 3$. If $k\geq 4$ then the edge is labelled with $k$; if $k=3$ the label is omitted;
\item an edge labelled with $\infty$ if $H_{e_i}$ and $H_{e_j}$ are parallel;
\item a dotted edge if $H_{e_i}$ and $H_{e_j}$ are ultraparallel. The dotted edge is labelled with the hyperbolic cosine of the length $l=d_{\mathcal{H}^n}(H_{e_i},H_{e_j})$ of their common perpendicular.
\end{enumerate}

\subsection{Commensurability and arithmeticity}\label{sec:commensur}

Let $H$ be a group. Two subgroups $H_1,H_2$ $\subset$ $H$ are \textit{commensurable} (in the wide sense) if and only if there exists an element $h\in H$ such that $H_1\cap h^{-1}H_2 h$ has finite index in both $H_1$ and $h^{-1}H_2 h$.

This notion defines an equivalence relation. In our context, the group $H$ will be $\Isom(\mathcal{H}^n)$. Stable under commensurability are some properties of subgroups of $\Isom(\mathcal{H}^n)$ such as discreteness, cofiniteness, cocompactness and arithmeticity. This latter notion can be further refined by splitting discrete subgroups in $\Isom(\mathcal{H}^n)$ into three categories: \textit{arithmetic}, \textit{quasi-arithmetic} and \textit{nq-arithmetic}. 

More precisely, let $K\subset\mathbb{R}$ be a totally real number field and let $V$ be a vector space of dimension $n+1$ over $K$ endowed with a quadratic form $q$ of signature $(n,1)$. We denote this quadratic space by $(V,q)$. Moreover for every non-trivial embedding $\sigma:K\hookrightarrow \mathbb{R}$ we assume that the quadratic space $(V,q^{\sigma})$ is positive definite, where $q^\sigma$ denotes the quadratic form obtained by applying $\sigma$ to each coefficient of $q$. Let
\[
\bigO(V,q):=\{U\in\GL(n+1,\mathbb{R})\mid q(Ux)=q(x) \text{ }\forall x\in V\otimes_K\mathbb{R} \}.
\]
Notice that since $q$ has the same signature and rank as the Lorentzian form $q_{-1}$, there exists a real invertible matrix $S$ such that $S^{-1}\bigO^+(n,1) S=\bigO^+(V,q)$.\\
Let $\mathcal{O}_K$ be the ring of integers of $K$. Consider a $\mathcal{O}_K$-lattice\footnote{Let $R$ be a ring with field of fraction $K$ and $V$ a vector space of dimension $n+1$ over $K$. An $R$-lattice $L$ in $V$ is an $R$-module in $V$ of rank $n+1$ for which $\Span_K \{L\}=V$.} $L$ and denote by $\bigO(L)<\bigO^+(V,q)\cap\GL(n+1,K)$ the group of linear transformations with coefficients in $K$ that preserve the lattice $L$. The group $\bigO(L)$ is discrete and of finite covolume \citep[\S2.2]{gromov1987non}.

Then, a discrete subgroup $\Gamma$ in $\Isom(\mathcal{H}^n)$ is called \textit{arithmetic of the simplest type} if there exist $K$, $q$ and $L$ as above such that $S^{-1}\Gamma S$ is commensurable with $\bigO (L)$ in $\bigO^+(V,q)\cap\GL(n+1,K)$. In this case one says that $\Gamma$ \textit{is defined over $K$ with quadratic space} $(V,q)$.

\begin{rmk}\label{rmk:existencearithm}\
 There is a more general definition of arithmetic group (see \citep[Chapter 6]{vinberg1988geometry}). However, if a hyperbolic Coxeter group is arithmetic, then it is of the simplest type \citep[Lemma 7]{vinberg1967discrete}. Since we will be working only with hyperbolic Coxeter groups, we will always refer to arithmetic groups of the simplest type as just \textit{arithmetic} groups.
\end{rmk}

More generally, a discrete subgroup $\Gamma$ in $\Isom(\mathcal{H}^n)$ is called \textit{quasi-arithmetic} if there exist $K$ and $q$ as above such that 
\[
S^{-1}\Gamma S\subset\bigO^+(V,q)\cap\GL(n+1,K).
\]
In this case one says that $\Gamma$ \textit{is defined over $K$ with quadratic space} $(V,q)$.

Finally, a discrete subgroup $\Gamma$ in $\Isom(\mathcal{H}^n)$ is called \textit{non-quasi-arithmetic}, \textit{nq-arithmetic} from now on, if it is neither arithmetic nor quasi-arithmetic.

We conclude this section by stating Vinberg's criterion to decide whether a hyperbolic Coxeter group is arithmetic, quasi-arithmetic or nq-arithmetic (see \citep[Theorem 2]{vinberg1967discrete}).

\begin{thm}[Vinberg's arithmeticity criterion]\label{thm:vinbergarithmeticity}
Let $\Gamma<\Isom(\mathcal{H}^n)$ be a Coxeter group of rank $N$ and denote by $G=(g_{ij})_{1\leq i,j\leq N}$ its Gram matrix. Let $\widetilde{K}$ be the field generated by the entries of $G$, and let $K(\Gamma)$ be the field generated by all the possible cycles of $G$. Then $\Gamma$ is quasi-arithmetic if and only if:
\begin{enumerate}[i)]
\item $\widetilde{K}$ is totally real;
\item for any embedding $\sigma:\widetilde{K}\hookrightarrow\mathbb{R}$ which is not the identity on $K(\Gamma)$, the matrix $G^{\sigma}$, obtained by applying $\sigma$ to all the coefficients of $G$, is positive semidefinite.
\suspend{enumerate}
Moreover, a quasi-arithmetic group $\Gamma$ is arithmetic if and only if
\resume{enumerate}[{[i)]}]
\item the cycles of $2\, G$ are algebraic integers in $K(\Gamma)$.
\end{enumerate}

In both cases, $\Gamma$ is defined over $K(\Gamma)$.

\end{thm}

\section{Commensurability of hyperbolic Coxeter groups}\label{commens}

In this section we prove the theorem stated in the Introduction and we show the commensurability property of the Vinberg ring. An important role will be played by the theory of fields of definition, which will be recalled in this section. 

\subsection{The Vinberg construction and fields of definition}\label{Vinconstruction}

We now associate a quadratic space to a hyperbolic Coxeter group following a construction due to Vinberg in \citep{vinberg1967discrete}. 

Let $\Gamma$ be a hyperbolic Coxeter group of rank $N$ and let $e_1,\dots, e_N\in\mathbb{R}^{n,1}$ be the outer normal unit vectors of its Coxeter polyhedron.
Let $G=(g_{ij})_{1\leq i,j\leq N}$ be the Gram matrix of $\Gamma$. For any $\{i_1, i_2,\dots ,i_l\}\subset\{1,2,\dots, N\}$ consider the cyclic product of $2\, G$
\begin{equation}\label{eq:cycles}
b_{i_1 i_2\dots i_l}:=2^l g_{i_1 i_2}g_{i_2 i_3}\dots g_{i_{l-1} i_{l}}g_{i_{l} i_{1}}. 
\end{equation}
Define the field $K(\Gamma):=\mathbb{Q}(\{b_{i_{1} i_{2}\dots i_{l}}\})$ of all cycles of $2\, G$. It is obvious that $K(\Gamma)$ is generated by the simple cycles.

Next, for $\{i_1,i_2,\dots , i_k\}\subset \{1,2,\dots ,N\}$, define the vectors
\begin{equation}\label{eq:vinbergvectors}
v_{1}:=2e_{1} \text{ and }v_{i_{1} i_{2} \dots i_{k}}:=2^k g_{1 i_{1}}g_{i_{1} i_{2}}\dots g_{i_{k-1}i_{k}}e_{i_{k}},
\end{equation}

and consider the $K(\Gamma)$-vector space $V$ spanned by the vectors $\{v_{i_{1} i_{2} \dots i_{k}}\}$ according to (\ref{eq:vinbergvectors}). By \citep[Lemma 1]{everitt2000constructing}, $V$ is of dimension $n+1$. Moreover, as shown in \citep{maclachlan1998invariant} for example, $V$ is left invariant by $\Gamma$ since
\begin{equation}\label{eq:vectorcalculI}
s_{e_j}(v_{i_1 i_2\dots i_k})=v_{i_1 i_2\dots i_k}-v_{i_1 i_2\dots i_k j},
\end{equation}
and
\begin{equation}\label{eq:vectorcalculII}
\langle v_{i_{1} i_{2} \dots i_{k}},v_{j_{1} j_{2} \dots j_{l}}\rangle \in K(\Gamma).
\end{equation}

Since $2\, G$ is of signature $(n,1)$, the restriction of the Lorentzian product on $V$ yields a quadratic form $q=q_V$ of signature $(n,1)$ on $V$. 

By combining the equations (\ref{eq:vectorcalculI}) and (\ref{eq:vectorcalculII}) a quick computation shows that
\begin{equation}\label{eq:3vincostr}
\langle s_{e_j}(v_{i_{1} i_{2} \dots i_{k}}),s_{e_j}(v_{j_{1} j_{2} \dots j_{l}})\rangle=\langle v_{i_{1} i_{2} \dots i_{k}},v_{j_{1} j_{2} \dots j_{l}}\rangle.
\end{equation}

By the construction of the $K(\Gamma)$-vector space $V$ in terms of the vectors (\ref{eq:vinbergvectors}) and the form $q_V$, we obtain a natural embedding $\Gamma\hookrightarrow \bigO(V,q)$. Observe that this construction is independent of the arithmetic nature of $\Gamma$. 

Therefore \textit{any} hyperbolic Coxeter group has an associated field and quadratic form which justifies the following definition.

\begin{defn}

Let $\Gamma$ be a hyperbolic Coxeter group. Then 
\begin{enumerate}[i)]
\item the field $K(\Gamma)=\mathbb{Q}(\{b_{i_{1} i_{2}\dots i_{l}}\})$ is called the \textit{Vinberg field} of $\Gamma$;
\item the quadratic form $q=q_V$ is called the \textit{Vinberg form} of $\Gamma$;
\item the quadratic space $(V,q)$ is called the \textit{Vinberg space} of $\Gamma$.
\end{enumerate}
\end{defn}

The next objective is to show that the Vinberg field and the similarity class of the Vinberg form are two commensurability invariants. Before that, we need more terminology.

Let $(V_1,q_1)$, $(V_2,q_2)$ be two quadratic spaces of dimension $m\geq 2$ over a field $K$. Then $(V_1,q_1)$ and $(V_2,q_2)$ are \textit{isometric} (denoted by $\cong$) if and only if there is an isomorphism $S:V_1\rightarrow V_2$ such that
\[
q_1(x)=q_2(Sx) \quad \forall x\in V_1.
\]
They are \textit{similar} (denoted by $\backsim$) if there exist a $\lambda\in K^{\ast}$ such that $(V_1,q_1)$ and $(V_2,\lambda q_2)$ are isometric. The scalar $\lambda$ is called \textit{similarity factor}.

Isometry and similarity induce equivalence relations. In the sequel, we often abbreviate and speak about \textit{isometric} (\textit{similar}) quadratic forms instead of isometric (similar) quadratic spaces. Furthermore, if one represents two quadratic forms by two $m\times m$ matrices $Q_1$ and $Q_2$ over $K$, then being isometric means that there exists an invertible matrix $S\in\GL(m,K)$ such that $Q_1=S^T Q_2 S$.

We conclude this part with some aspects about fields of definition, which will play an essential role for the upcoming proofs. All the theory presented here is taken from Vinberg's paper \cite{vinberg1971rings}.

Let $H$ be a Lie group. The \textit{adjoint trace field} $\mathbb{Q}(\Tr\Ad H)$ of $H$ is defined as the field generated by the traces of the adjoint representation of the elements of $H$, namely
\[
\mathbb{Q}(\Tr\Ad H)=\mathbb{Q}(\Tr\Ad(h) \mid h\in H)
.\]

Let $U$ be a finite dimensional vector space over a field $F$, and let $R\subset F$ be an integrally closed Noetherian ring. Denote by $\Delta$ a family of linear transformations of $U$. The ring $R$ is said to be a \textit{ring of definition} for $\Delta$ if $U$ contains an $R$-lattice which is invariant under $\Delta$. When $R$ is a field we call $R$ a \textit{field of definition}.

If a principal ideal domain $R$ is a ring of definition for $\Delta$, then we can find a basis of $U$ such that every element of $\Delta$ can be written as a matrix having entries in $R$.

Let us specialise the context and consider a Coxeter group $\Gamma<\Isom(\mathcal{H}^n)$. As we have seen in Section \ref{Vinconstruction}, the space $\mathbb{R}^{n,1}$ contains the $K(\Gamma)$-module $V$ which is invariant under $\Gamma$. That is, the Vinberg field $K(\Gamma)$ is a field of definition for $\Gamma$. The next lemma implies that the Vinberg field $K(\Gamma)$ is actually the \textit{smallest} field of definition associated to $\Gamma$. 

\begin{lemma}[\citep{vinberg1971rings}, Lemma 11 and Lemma 12]\label{lemmacicli}
Let $\Gamma$ be a hyperbolic Coxeter group with Gram matrix $G$ and let $F$ be a field of characteristic $0$. An integrally closed Noetherian ring $R\subset F$ is a ring of definition for $\Gamma$ if and only if $R$ contains all the simple cycles of $2\, G$.
\end{lemma}

Lastly, for the following proofs we need the result \cite[Theorem 5]{vinberg1971rings} of Vinberg. We recapitulate here a more specific version suitable to our context.
\begin{thm}\label{thmVin}
Let $\Gamma$ be a cofinite hyperbolic Coxeter group with Vinberg space $(V,q)$ and Gram matrix $G$. Let $R$ be an integrally closed Noetherian ring. Then the following is equivalent:
\begin{enumerate}[i)]
\item $R$ is a ring of definition of $\Gamma$,
\item $R$ is a ring of definition of $\Ad\Gamma$,
\item $R$ contains all the simple cyclic products of $2\, G$.
\end{enumerate}
 
\end{thm}

\begin{rmk}\label{rmk:Choi}
It is important to notice that in \cite{vinberg1971rings} Vinberg considers Zariski dense groups generated by reflections of a quadratic space defined over an \textit{algebraically closed} field. This hypothesis does not apply directly to our situation since the isometry group $\PO(n,1)$ of Klein's projective model $\mathcal{K}^n$ is defined over the reals. 

Our version of the theorem can be retrieved from the original one as follows. Pass to the complexified space $\mathbb{R}^{n+1}\otimes_\mathbb{R}\mathbb{C}$ endowed with the standard (real) Lorentzian form $q_{-1}$. Let $\bigO_{\mathbb{C}}(n,1)$ be the group of complex $(n+1) \times (n+1)$ matrices which preserve $q_{-1}$, and form the projective group $\PO_\mathbb{C}(n,1)=\bigO_{\mathbb{C}}(n,1)/\{\pm I\}$. 

Recall that a cofinite hyperbolic Coxeter group is Zariski dense (over $\mathbb{R}$) in $\PO(n,1)$ (see \citep[Chapter 4]{kim2004rigidity}). This property remains valid in the complexified context of $\PO_\mathbb{C}(n,1)$ over $\mathbb{C}$. We can now apply the original Theorem 5 of \citep{vinberg1971rings} which implies Theorem \ref{thmVin}.

\end{rmk}

\subsection{Commensurability conditions for hyperbolic Coxeter groups}\label{Newcriterion}

We are now able to prove the theorem stated in the Introduction. 

\begin{thm}\label{thetheorem}
Let $\Gamma_1$ and $\Gamma_2$ be two commensurable cofinite hyperbolic Coxeter groups acting on $\mathcal{H}^n$, $n\geq 2$. Then their Vinberg fields coincide and the two associated Vinberg forms are similar over this field.
\end{thm}

We start the proof by showing that two commensurable Coxeter groups have the same Vinberg field.

\begin{prop}\label{firstlemma}
Let $\Gamma <\Isom(\mathcal{H}^n)$ be a cofinite Coxeter group, $n\geq 2$. Then the associated Vinberg field and the adjoint trace field coincide, that is
\begin{equation}\label{equationadjoint}
K(\Gamma)=\mathbb{Q}(\Tr \Ad\Gamma).
\end{equation}
\end{prop}

\begin{proof}

Lemma \ref{lemmacicli} implies that the Vinberg field $K(\Gamma)$ is the smallest field of definition of $\Gamma$. By point $i)$ of Theorem \ref{thmVin}, $K(\Gamma)$ is a field of definition of $\Ad\Gamma$ as well and by point $iii)$ $K(\Gamma)$ is contained in every field of definition of $\Ad\Gamma$. By \citep[Corollary of Theorem 1]{vinberg1971rings}, $\mathbb{Q}(\Tr \Ad\Gamma)$ is the smallest field of definition of $\Ad\Gamma$. Hence, the equality (\ref{equationadjoint}) follows.

\end{proof}

\begin{cor}\label{mainthmpartI}
Let $\Gamma_1$, $\Gamma_2 < \Isom(\mathcal{H}^n)$ be two cofinite Coxeter groups, $n\geq 2$. If $\Gamma_1$ and $\Gamma_2$ are commensurable, then their associated Vinberg fields coincide, that is,
\[
K(\Gamma_1)=K(\Gamma_2).
\]
\end{cor}

\begin{proof}
By Proposition \ref{firstlemma} we know that $K(\Gamma_1)=\mathbb{Q}(\Tr\Ad\Gamma_1)$ and $K(\Gamma_2)=\mathbb{Q}(\Tr\Ad\Gamma_2)$. The adjoint trace field of a hyperbolic lattice  is a commensurability invariant (see \cite[Proposition 12.2.1]{deligne1986monodromy}). Therefore the claim follows. 
\end{proof}

\begin{rmk}\label{rmk:trace}\
\begin{enumerate}[i)]
\item Reflections in $\bigO^+(n,1)$ have traces equal to $n-1$. Thus, by Corollary $1$ of Theorem $4$ of \citep{vinberg1971rings}, the smallest field of definition of $\Gamma$ is $\mathbb{Q}(\Tr\Gamma)$. Hence we also get the equality $K(\Gamma)=\mathbb{Q}(\Tr \Gamma)$.
\item By the Local Rigidity Theorem \cite[Chapter 1]{vin2000liegroups}, the adjoint trace field $\mathbb{Q}(\Tr \Ad \Gamma)$ of a Coxeter group in $\Isom (\mathcal{H}^n)$ is a number field for $n\geq 4$. Therefore, by Proposition \ref{firstlemma}, the Vinberg field $K(\Gamma)$ is a number field. Moreover, $K(\Gamma)$ is a number field for $n=3$ as well. This is a consequence of the connection between $K(\Gamma)$ and the invariant trace field $K\Gamma^{(2)}$ (\citep[Theorem 3.1]{maclachlan1998invariant}) and the fact that $K\Gamma^{(2)}$ is a number field (\citep[Theorem 3.1.2]{maclachlan2013arithmetic}).
\end{enumerate}
\end{rmk}

\begin{ex}
Consider the two non-cocompact nq-arithmetic Coxeter pyramid groups $\Gamma_1$ and $\Gamma_2$ acting on $\mathcal{H}^4$ as shown in Figure \ref{threepyra}.

\begin{figure}[h!]
\centering
\captionsetup[subfloat]{position=bottom,labelformat=empty}
\subfloat[$\Gamma_1$]{
\begin{tikzpicture}[plane/.style={
                      circle, fill=black, minimum size=5pt, inner sep=0pt
                    }, font=\footnotesize
                   ]

	  \draw (5,0) node[plane]{} -- (6,0) node[midway, above]{$6$} node[plane]{}
        -- (7,0) node[plane]{}--(8,0)node[plane]{}
        -- (9,1)node[midway, above]{$4$} node[plane]{} -- (9,-1) node[midway, right]{$\infty$} node[plane]{} -- (8,0);
	
	\end{tikzpicture}
	$\qquad$
	}
	\subfloat[$\Gamma_2$]{
\begin{tikzpicture}[plane/.style={
                      circle, fill=black, minimum size=5pt, inner sep=0pt
                    }, font=\footnotesize
                   ]	
	
		  \draw (11,0) node[plane]{} -- (12,0) node[midway, above]{$6$} node[plane]{}
        -- (13,0)  node[plane]{}--(14,0)node[plane]{}
        -- (15,1)node[midway, above]{$5$} node[plane]{} -- (15,-1) node[midway, right]{$\infty$} node[plane]{}-- (14,0);

	\end{tikzpicture}
}
\text{  }

  \caption{Two Coxeter pyramid groups $\Gamma_1$ and $\Gamma_2$ in $\Isom(\mathcal{H}^4)$.}
  \label{threepyra}
\end{figure}
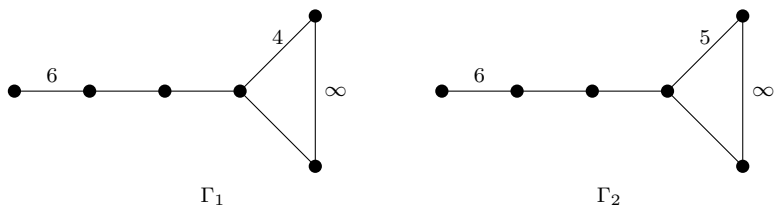

For $\Gamma_1$ we have the cycle $-2\sqrt{2}$ obtained by following the triangular tail path, while all the other simple cycles lie in $\mathbb{Q}$ or $\mathbb{Q}\left(\sqrt{2}\right)$. For $\Gamma_2$, the triangular tail path gives the cycle $\frac{1}{2}\left(3+\sqrt{5}\right)$. All other simple cycles are either in $\mathbb{Q}$ or $\mathbb{Q}\left(\sqrt{5}\right)$. Therefore the Vinberg fields are $K(\Gamma_1)=\mathbb{Q}\left(\sqrt{2}\right)$ and $K(\Gamma_2)=\mathbb{Q}\left(\sqrt{5}\right)$. Thus $\Gamma_1$ and $\Gamma_2$ are incommensurable.\\ 
\end{ex}

Let us return to the proof of Theorem \ref{thetheorem} and show that two commensurable hyperbolic Coxeter groups have similar Vinberg forms. The proof will follow the same strategy as indicated by Gromov and Piateski-Shapiro in Theorem 2.6 of \citep{gromov1987non} for arithmetic groups, and which has been elaborated by Johnson, Kellerhals, Ratcliffe and Tschantz in Theorem 1 of \citep{johnson2002commensurability} for the special case of hyperbolic Coxeter simplex groups.

Consider two commensurable hyperbolic Coxeter groups $\Gamma_1$ and $\Gamma_2$ represented in $\bigO^+(n,1)$ and denote their Vinberg field by $K$. There is a matrix $X\in\bigO^+(n,1)$ and there are two subgroups $H_1<\Gamma_1$ and $H_2<\Gamma_2$, each of finite index, such that $H_1=X^{-1}H_2 X$. 
One can assume that $H_1$ and $H_2$ are contained in $\SbigO^+(n,1)$, the index two subgroup of $\bigO^+(n,1)$ of determinant one matrices. Let $(V_1,q_1)$ be the Vinberg space over $K$ associated to $\Gamma_1$ and equipped with the basis $\{v_1,\dots ,v_{n+1}\}$ according to \eqref{eq:vinbergvectors}. With respect to this basis, all elements of $\Gamma_1$ are matrices over $K$, since $K$ is a field of definition of $\Gamma$. Clearly the forms $q_1$ and $q_{-1}$ are equivalent over $\mathbb{R}$. The same reasoning applies to the Vinberg space $(V_2,q_2)$. Let $Q_1$ and $Q_2$ be the matrix representations of the Vinberg forms $q_1$ and $q_2$ in the relative bases. Let the real matrices $T_1$ and $T_2$ be the representations of the isometries between the Vinberg forms and the Lorentzian form $q_{-1}$. Then the matrix 
\begin{equation}\label{eq:transictionmatrix}
S:=T_2^{-1}X T_1
\end{equation}
represents an isometry between $q_1$ and $q_2$, since $Q_1=S^{T} Q_2 S$. Moreover define the two groups $H'_1:=T_1^{-1}H_1 T_1$ and $H'_2:=T_2^{-1}H_2 T_2$.

Consider the isomorphism between the orthogonal groups $\bigO(q_1)$ and $\bigO(q_2)$ given by 
\begin{equation}\label{eq:defnphi}
\phi: A\rightarrow SA S^{-1}. 
\end{equation}

\begin{lemma}\label{isorestrict}
The map $\phi$ restricts to a $K$-linear map on $\Mat(n+1,K)$.
\end{lemma}

\begin{proof}

Let $i\in\{1,2\}$. Denote by $\bigO^+(q_i)$ the group of $q_i$-orthogonal maps which leave each sheet of the hyperboloid $\mathcal{H}_i^{n+1}=\{x\in\mathbb{R}^{n+1}\mid q_i(x)=-1\}$ invariant. The isometry between $q_i$ and $q_{-1}$ gives a group isomorphism between $\bigO(q_i)$ and $\bigO(q_{-1})$. This isomorphism maps $\bigO^+(q_i)$ onto $\bigO^+(n,1)$. Analogously, $\SbigO^+(q_i)$ is mapped onto $\SbigO^+(n,1)$ and hence $H_i'\subset\SbigO^+(q_i)$. Now, $\SbigO^+(n,1)$ is a non-compact connected simple Lie group and thus the same can be said for $\SbigO^+(q_i)$. Since $H_i'$ has finite covolume, by the Borel density theorem \citep{borel1960density} we get that $\Span_{\mathbb{R}}(H_i')=\Span_{\mathbb{R}}(\SbigO^+(q_i))$ in $\Mat(n+1,\mathbb{R})$. Furthermore, the action of $\SbigO^+(n,1)$ on $\mathbb{C}^{n+1}$ is irreducible\footnote{See the Erratum to the paper ``Commensurability classes of hyperbolic Coxeter groups", due to J. Ratcliffe and S. Tschantz, presented in \citep[Appendix D]{dotti2020}.}, and hence that the action of $\SbigO^+(q_i)$ is irreducible as well. By Burnside's theorem \citep{burnside1905condition} (see also \citep{lam1998theorem}) we get the equality $\Span_{\mathbb{R}}(\SbigO^+(q_i))=\Mat(n+1,\mathbb{R})$, which implies that $\Span_{\mathbb{R}}(H_i')=\Mat(n+1,\mathbb{R})$. 

Notice that for each $\alpha\in K$ and $C\in\Mat(n+1,K)$ we have $\phi(\alpha C)=\alpha\phi(C)$. Recall that $K$ is a field of definition for $H_i$, thus $H_i'\subset\Mat(n+1,K)$. By the same arguments as before, we have that $\Span_{K}(H_i')=\Mat(n+1,K)$. Moreover, by (\ref{eq:transictionmatrix}),
\[
\phi(H'_1)=\phi(T_1^{-1}H_1 T_1)=T^{-1}_2 X H_1 X^{-1} T_2=T^{-1}_2 H_2  T_2=H'_2.
\]
We deduce that $\phi(\Span_{K}(H_1'))=\Span_{K}(H_2')$. Therefore $\phi$ restricts to a $K$-linear map on $\Mat(n+1,K)$.
\end{proof}

Based on Lemma \ref{isorestrict} we are finally ready to prove the last step as given by the following proposition. Its proof is a direct adaptation of the corresponding step in the proof of \citep[Theorem 1]{johnson2002commensurability}. 

\begin{prop}\label{prop:similarityfactor}
Let $\Gamma_1$, $\Gamma_2$ be two commensurable Coxeter groups in $\Isom(\mathcal{H}^n)$, $n\geq 2$, with Vinberg field $K(\Gamma_1)=K(\Gamma_2)=:K$. Then the two Vinberg forms $q_1$ and $q_2$ are similar over $K$. Moreover, the similarity factor is positive.
\end{prop}

\begin{proof}

Let $1\leq i,j \leq n+1$. Define the matrix $I_{ij}\in\Mat(n+1,K)$ with coefficient $[I]_{ij}=1$ and all the other coefficients equal to zero. Consider the isomorphism $\phi$ according to \eqref{eq:defnphi}. Define $M_{ij}:=\phi(I_{ij})=S I_{ij}S^{-1}$, which is in $\Mat(n+1,K)$ by Lemma \ref{isorestrict}. The matrix $S I_{ij}=:S_{ij}$ has the $j$-th column which is equal to the $i$-th column of $S$ and all the other coefficients are equal to zero. Observe that $[M_{ij}]_{kl}=[S]_{ki}[S^{-1}]_{jl}$ for all $k,l,i,j$. The matrix $S^{-1}$ is invertible, thus we can always find a pair $\{j,l\}$ such that $[S^{-1}]_{jl}\neq 0$. Let $\lambda$ denote the inverse of the coefficient $[S^{-1}]_{jl}$. In doing so, every coefficient of $S$ can be written as $\lambda$ multiplied with an entry of a matrix of the form $M_{ij}\in\Mat(n+1,K)$. Hence there exists a matrix $M\in\Mat(n+1,K)$ such that $S=\lambda M$. Recall that $Q_1=S^{T} Q_2 S$ holds, therefore $Q_1=\lambda^2 M^T Q_2 M$, with  $Q_1$, $Q_2$ and $M$ all in $\Mat(n+1,K)$. Finally $\lambda^2$ is a positive element belonging to $K$ so that $q_1$ is isometric to $\lambda^2 q_2$, and the claim follows
\end{proof}

\subsection{Similarity classification of the Vinberg forms}\label{sec:similarityvinbergform}

The study of similarity of quadratic forms heavily relies on isomorphisms of quaternion algebras and others elements of the Brauer group. For a more detailed explanation on this topic we refer to \citep{kellerhals2017commensurability}. Let $K$ denotes a field of characteristic different from $2$, and let $q$ be a quadratic form of dimension $m$ over $K$, that is, $q$ is defined on a vector space of dimension $m$ over $K$.

For two elements $a$, $b\in K^{\ast}$, we denote by $(a,b)$ the quaternion algebra over $K$ generated by the elements $1$, $i$, $j$, $ij$ with the relations $i^2=a$, $j^2=b$ and $ij=-ji$.

Two quaternion algebras are said to be equivalent if and only if they are isomorphic. Equivalence classes of quaternion algebras form a group, which is a subgroup of the Brauer group $\Br(K)$. For some computational rules about the multiplication between quaternion algebras, we refer to \citep[Proposition 3.20]{lam2005introduction}. If the field $K$ is a number field, then two quaternion algebras over $K$ are isomorphic if and only if they have the same ramification set (see \citep[Theorem 4.1]{maclachlan2011commensurability}).

For the definition of a ramification set $\Ram(A)$ for a quaternion algebra $A$ and its theory we refer to \citep{maclachlan2011commensurability}. In this paper, the computations of ramification sets are done using the package RamifiedPlaces of Magma\textsuperscript{\textcopyright}.

The similarity classification of quadratic forms relies on two elements of the Brauer group, which are closely related to one another. The first one the \emph{Hasse invariant} $s(q)$ of a diagonal quadratic form $q=\langle a_1,\ldots,a_m\rangle$. This is the element of the Brauer group $\Br(K)$ represented by the quaternion algebra
	\[
		s(q)=\bigotimes_{i<j} (a_i,a_j)_K\,.
	\]
The Hasse invariant $s(q)$ is independent of the diagonalisation chosen. It is moreover an isometry invariant (see \citep[Proposition 3.18]{lam2005introduction}). However it is \textit{not} a similarity invariant (see \cite[Lemma 4.3]{meyer2017totally}). 

The second element of the Brauer group we are interested in is the \textit{Witt invariant} $c(q)$ of a quadratic space $(V,q)$ over $K$. It is obtained from the Hasse invariant according to \citep[Chapter V, Proposition 3]{lam2005introduction}.

Let $\Gamma$ be a quasi-arithmetic Coxeter group with Vinberg field $K$ acting on $\mathcal{H}^n$. Let $(V,q)$ be the Vinberg space of dimension $n+1$ over $K$ associated to $\Gamma$ and put $\delta :=(-1)^{\frac{n(n+1)}{2}}\det (q)$, the \textit{discriminant} of $q$. Denote by $B$ the quaternion algebra representing the Witt invariant $c(q)$. The similarity class of $(V,q)$ depends on the parity of $n$ as follows.

\begin{thm}[\citep{maclachlan2011commensurability}, Theorem 7.2]\label{macla1}
When $n$ is even, the similarity class of the Vinberg space $(V,q)$ of dimension $n+1$ is in one-to-one correspondence with the isomorphism class of the quaternion algebra $B$.
\end{thm}

\begin{thm}[\citep{maclachlan2011commensurability}, Theorem 7.4]\label{macla2}
When $n$ is odd, the similarity class of the Vinberg space $(V,q)$ of dimension $n+1$ is in one-to-one correspondence with the isomorphism class of the quaternion algebra $B\otimes_K K(\sqrt{\delta})$ over $K(\sqrt{\delta})$. Moreover, if $\delta$ is a square in $K^{\ast}$, then the similarity class is in one-to-one correspondence with the isomorphism class of $B$ over $\mathbb{Q}$.
\end{thm}

Consider now a nq-arithmetic group $\Gamma$ acting on $\mathcal{H}^n$. For $n$ even, a similarity criterion can be stated. For $n$ odd, we provide a necessary condition for similarity, only. We start by recalling the Hasse-Minkowski Theorem in terms of the Hasse invariant see (\citep{lam2005introduction}, \citep{bayer11}). 
\begin{thm}\label{thm:hassemink}
Let $K$ be a number field and let $q_1$ and $q_2$ be two quadratic forms of dimension $m$ over $K$. For a $\lambda\in K^{\ast}$, $q_1$ and $\lambda q_2$ are isometric if and only if the following properties are satisfied:   

\begin{enumerate}[i)]
\item $\dim(q_1)=\dim(\lambda q_2)$,

\item $\det(q_1)\equiv\det(\lambda q_2)$ in $K^{\ast}\mod (K^{\ast})^2$,

\item $s(q_1)=s(\lambda q_2)$,

\item $\sgn(\sigma(q_1))=\sgn(\sigma(\lambda q_2))$ for all real embeddings $\sigma: K\hookrightarrow\mathbb{R}$.
\end{enumerate}
\end{thm}

For $n$ even, let $\Gamma_1$ and $\Gamma_2$ be two hyperbolic Coxeter groups with the same Vinberg field $K$, and denote by $q_1$ and $q_2$ the associated Vinberg forms over $K$. Recall that $\dim(q_1)=\dim(q_2)=n+1=:m$, i.e. the dimension of both quadratic forms is odd.
Then, condition $ii)$ of the Hasse-Minkowski Theorem \ref{thm:hassemink} implies that $\det(q_1)\equiv\lambda\det(q_2)$ in $K^{\ast}\mod (K^{\ast})^2$. This means that $\lambda$ can only be the value which balances the two determinants, that is, $\lambda=\frac{\det(q_1)}{\det(q_2)}\in K^{\ast}/(K^{\ast})^2$ (see also the proof of \cite[Proposition 5.4]{meyer2017totally}).
Using \cite[Lemma 4.3]{meyer2017totally}, we can simplify the Hasse invariant $s(\lambda q_2)$ for $\lambda=\frac{\det(q_1)}{\det(q_2)}\in K^{\ast}/(K^{\ast})^2$, and we obtain the complete set of similarity invariants for Vinberg forms as shown in Table \ref{figure:similarity:completeinvariant}.

\begin{table}[h!]
	\centering
	\renewcommand{\arraystretch}{1.3}
	\begin{tabular}{l|l}
		$n$ & Similarity criterion \\ \hline \hline
		\multirow{2}{*}{$n\equiv 0\mod 4$} & $s(q_1)=s(q_2)$ \\
		&  $\sgn(\sigma(q_1))=\sgn(\sigma(\lambda q_2))$ \\ \hline
		\multirow{2}{*}{$n\equiv 2\mod 4$} & $s(q_1)=(\lambda,-1)\cdot s(q_2)$ \\
		&  $\sgn(\sigma(q_1))=\sgn(\sigma(\lambda q_2))$ \\ \hline
	\end{tabular}
	\caption{Similarity criterion for Vinberg forms of hyperbolic Coxeter groups.}
	\label{figure:similarity:completeinvariant}
\end{table}

Notice that this similarity classification is compatible with the one provided by Maclachlan for quasi-arithmetic groups. For these groups, the equality between signatures is always satisfied.

For $n$ odd, let $\Gamma_1$ and $\Gamma_2$ be two hyperbolic Coxeter groups. If they are quasi-arithmetic, we refer to the similarity classification provided by Maclachlan (see Theorem \ref{macla2}). Otherwise, the similarity problem for their even-dimensional Vinberg forms $q_1$ and $q_2$ is more involved. We present here a partial result, only.

Applying condition $ii)$ of the Hasse-Minkowski Theorem \ref{thm:hassemink} we get $\det(q_1)\equiv\det(\lambda q_2)$ in $K^{\ast}\mod (K^{\ast})^2$ which reduces to $\det(q_1)\equiv\det(q_2)\mod (K^{\ast})^2$. In contrast to the previous case, we can not extract any information about $\lambda$. This fact can be stated in the following lemma, sometimes referred to as the \textit{ratio-test}.

\begin{lemma}
Let $\Gamma_1$, $\Gamma_2 < \Isom(\mathcal{H}^n)$, $n$ odd, be two commensurable Coxeter groups with Vinberg field $K$ and Vinberg forms $q_1$ and $q_2$, respectively. Then, $\det(q_1)\equiv\det(q_2)\in K^{\ast}\mod (K^{\ast})^2$.
\end{lemma}

\begin{ex}
As an incommensurability example using the Vinberg form, consider the two cocompact Coxeter groups $\Gamma_1$, $\Gamma_2$ in $\Isom(\mathcal{H}^4)$ given in Figure \ref{Napierex}. The groups $\Gamma_1$ and $\Gamma_2$ are so-called crystallographic Napier cycles (see \cite{im1990napier}). Observe that both groups are quasi-arithmetic (but not arithmetic).

\begin{figure}[h!]
  \centering
  \captionsetup[subfloat]{position=bottom,labelformat=empty}
  \subfloat[$\Gamma_1$]{
    \begin{tikzpicture}[plane/.style={ circle, fill=black, minimum size=5pt, inner sep=0pt }, font=\footnotesize ]
      \draw[dashed] (-2,1) node[plane]{} -- (0,1)node[midway, above]{$l_2$}  node[plane]{}--(1,0)node[midway, above]{$l_3$} node[plane]{};
      \draw[dashed] (-3,0) node[plane]{} -- (-2,1)node[midway, above]{$l_1$};
      \draw(1,0)  node[plane]{}--(1,-1.5)node[midway, right]{}node[plane]{};
      \draw(1,-1.5) node[plane]{} -- (-1,-2.5) node[midway, below]{$5$}node[plane]{} -- (-3,-1.5) node{} node[plane]{};
      \draw(-3,-1.5) node[plane]{} -- (-3,0)node[midway, left]{$4$};
    \end{tikzpicture}
  }
  \subfloat[$\Gamma_2$]{
    \begin{tikzpicture}[plane/.style={ circle, fill=black, minimum size=5pt, inner sep=0pt }, font=\footnotesize ]
      \draw[dashed] (1,1) node[plane]{} -- (3,1)node[midway, above]{$l'_2$}  node[plane]{}--(4,0)node[midway, above]{$l'_3$} node[plane]{};
      \draw[dashed] (0,0) node[plane]{} -- (1,1)node[midway, above]{$l'_1$} ;
      \draw(4,0)  node[plane]{}--(4,-1.5)node[midway, right]{}node[plane]{};
      \draw(4,-1.5) node[plane]{} -- (2,-2.5) node[midway, below]{$5$}node[plane]{} -- (0,-1.5) node{} node[plane]{};
      \draw(0,-1.5) node[plane]{} -- (0,0)node[midway, left]{$5$};
    \end{tikzpicture}
  }
  \caption{The Coxeter groups $\Gamma_1$ and $\Gamma_2$ in $\Isom(\mathcal{H}^4)$.}
  \label{Napierex}
\end{figure}
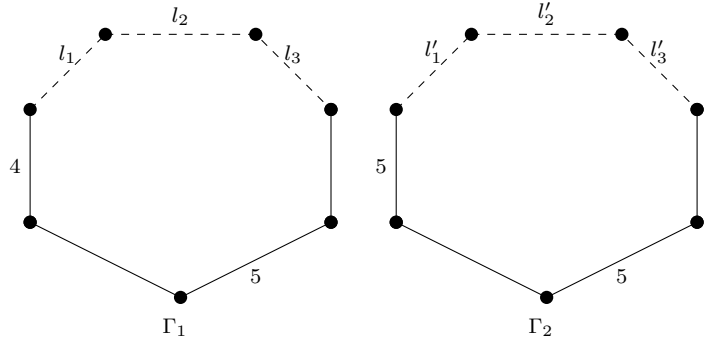

The weights $l_i$ and $l_i'$ of the dotted edges in the Coxeter graphs can be computed and are
\begin{align*}
  l_1 & = \sqrt{\frac{1}{11}\left(10+3\sqrt{5}\right)}, & l'_1 & = \sqrt{\frac{2}{11}\left(7+\sqrt{5}\right)}, \\[1ex]
  l_2 & = \frac{1}{2}\sqrt{\left(5+\sqrt{5}\right)}, & l'_2 & = \sqrt{\frac{2}{19}\left(9+\sqrt{5}\right)}, \\[1ex]
  l_3 & = \sqrt{\frac{1}{11}\left(16+7\sqrt{3}\right)}, & l'_3 & = \sqrt{\frac{1}{209}\left(233+104\sqrt{5}\right)}.
\end{align*}

The groups $\Gamma_1$ and $\Gamma_2$ have both $K=\mathbb{Q}\left(\sqrt{5}\right)$ as their Vinberg field. The diagonalised associated Vinberg forms over $K$ are
\begin{align*}  
              q_1= & \diag\left(4,4,4,-2-2\sqrt{5},20+8\sqrt{5}\right),\\
              q_2= & \diag\left(4,\frac{5}{2}+\frac{1}{2}\sqrt{5},2+\frac{2}{5}\sqrt{5},\frac{-37}{2}-\frac{17}{2}\sqrt{5},\frac{312}{19}+\frac{136}{19}\sqrt{5}\right)          
.
\end{align*}

These forms have the following Hasse invariants:
\begin{align*}  
              c(\Gamma_1)= & \left(-2-2\sqrt{5},5+2\sqrt{5}\right),\\
              c(\Gamma_2)= & \left(10+2\sqrt{5},-1\right)\cdot\left(-74-34\sqrt{5},1482+646\sqrt{5}\right).
\end{align*}
The ramification set $\Ram(\Gamma_1)$ contains two prime ideals, one generated by $2$, and the other generated by $5$ and $-1+2\sqrt{5}$. The ramification set $\Ram(\Gamma_2)$ is empty. Since $\Ram(\Gamma_1)\neq\Ram(\Gamma_2)$, the two quaternion algebras representing $c(\Gamma_1)$ and $c(\Gamma_2)$ are not isomorphic. Hence the Vinberg forms $q_1$ and $q_2$ are not similar, and the groups $\Gamma_1$ and $\Gamma_2$ are incommensurable.
\end{ex}

\subsection{The Vinberg ring}\label{vinbergring}

In this section we are looking for additional commensurability invariants for arbitrary hyperbolic Coxeter groups.

\begin{defn}
Let $\Gamma <\Isom(\mathcal{H}^n)$, $n\geq 2$, be a cofinite Coxeter group with Gram matrix $G$. Consider all the cycles $b_{i_1 i_2\dots i_l}=2^l g_{i_1 i_2}g_{i_2 i_3}\dots g_{i_{l-1} i_{l}}g_{i_{l} i_{1}}$ of $2 G$. The ring 
\[
R(\Gamma):=\mathcal{O}(\{b_{i_1 i_2\dots i_l}\})
\]
is called the \textit{Vinberg ring} of $\Gamma$.
\end{defn}

We show that the Vinberg ring is a ring of definition for certain groups and hence a commensurability invariant. Notice that the Vinberg ring as commensurability invariant is superfluous when considering arithmetic groups.

\begin{prop}\label{prop:vinbergringcomminv}
Let $\Gamma <\Isom(\mathcal{H}^n)$, $n\geq 2$, be a cofinite Coxeter group with Gram matrix $G$. Assume that its Vinberg field $K$ is a number field. Then the Vinberg ring $R(\Gamma)$ is a commensurability invariant. 
\end{prop}

\begin{proof}
For this proof we use some results of Davis about overrings\footnote{An overring of an integral domain is a subring of the quotient field containing that given ring. In our case, the integral domain is the ring of integers $\mathcal{O}$ of the Vinberg field $K$, which has the Vinberg field as its quotient field.} (\citep{davis1964overrings}) in the same way as used by Mila in \citep[Section 2.1]{mila2018nonarithmetic}. Since by hypothesis the Vinberg field $K$ is a number field, there exists a minimal ring of definition $R$ for $\Gamma$ (\citep[Corollary to Theorem 1]{vinberg1971rings}) which equals the integral closure of $\mathbb{Z}[\Tr\Ad\Gamma]$ in $K$. Clearly $R$ is integrally closed and therefore contains the ring of integers $\mathcal{O}$ of $K$. Thus $R$ is the integral closure of $\mathcal{O}[\Tr\Ad\Gamma]=:R'$ in $K$. The ring $R'$ is an overring of $\mathcal{O}$. Since $\mathcal{O}$ is a Noetherian integral Dedekind domain, by \citep[Theorem 1]{davis1964overrings} its overring $R'$ is integrally closed. This implies $R=R'$ which means that $\mathcal{O}[\Tr\Ad\Gamma]$ is the smallest ring of definition for $\Gamma$. Moreover, the Vinberg ring $R(\Gamma)$ is also an overring of $\mathcal{O}$ in such a way that it is integrally closed as well, and it is furthermore Noetherian since it is a subring of the number field $K$ (see \citep[Theorem]{gilmer1970integral}). Hence $R(\Gamma)$ is a ring of definition for $\Gamma$. By Theorem \ref{thmVin}, $R(\Gamma)\subset R'$. Now, $R'$ is the smallest ring of definition so that $R(\Gamma)=\mathcal{O}[\Tr\Ad\Gamma]$. By Theorem 3 of \citep{vinberg1971rings}, rings of definition are commensurability invariants. Thus $R(\Gamma)$ is a commensurability invariant.
\end{proof}

\begin{ex}

As an example, consider the two non-cocompact quasi-arithmetic (but not arithmetic) Coxeter cube groups $\Gamma_1$ and $\Gamma_2$ in $\Isom(\mathcal{H}^3)$ defined in Figure \ref{idealcubering} (see \citep{jacquemet2017hyperbolic}). They both have $\mathbb{Q}$ as Vinberg field and similar quadratic forms. Their Vinberg rings are given by $R(\Gamma_1)=\mathbb{Z}[1/3]$ and $R(\Gamma_2)=\mathbb{Z}[1/2]$, respectively. By Proposition \ref{prop:vinbergringcomminv} the groups $\Gamma_1$ and $\Gamma_2$ are therefore incommensurable.
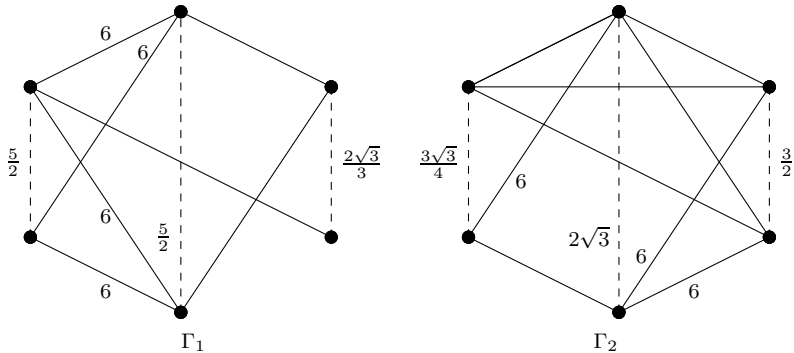
\begin{figure}[h!]	
\centering
\captionsetup[subfloat]{position=bottom,labelformat=empty}
\subfloat[$\Gamma_1$]{
\begin{tikzpicture}[plane/.style={
                      circle, fill=black, minimum size=5pt, inner sep=0pt
                    }, font=\footnotesize
                   ]
  \draw (-3,0) node[plane]{} -- (-1,1) node[midway, above]{$6$} node[plane]{}--(1,0) node[plane]{};
        \draw[dashed] (1,0)  node[plane]{}--(1,-2)node[midway, right]{$\frac{2\sqrt{3}}{3}$}node[plane]{};
      \draw  (-1,-3) node[plane]{} -- (-3,-2) node[midway, below]{$6$} node[plane]{};
      \draw [dashed](-3,-2) node[plane]{} -- (-3,0)node[midway, left]{$\frac{5}{2}$};
      \draw (-3,0) node[plane]{}--(-1,-3)node[midway, below]{$6$}node[plane]{}--(1,0)node[plane]{};
      \draw (1,-2)node[plane]{}--(-3,0)node[plane]{};
      \draw (-3,-2)node[plane]{}--(-1,1)node[near end, above]{$6$} node[plane]{};
      
     \draw [dashed](-1,-3)node[plane]{}--(-1,1)node[near start, left]{$\frac{5}{2}$}node[plane]{};
	\end{tikzpicture}
	}
	\subfloat[$\Gamma_2$]{
	\begin{tikzpicture}[plane/.style={
                      circle, fill=black, minimum size=5pt, inner sep=0pt
                    }, font=\footnotesize
                   ]
	
  \draw (5,0) node[plane]{} -- (7,1)  node[plane]{}
        -- (9,0)  node[plane]{};
        \draw (9,0) [dashed]  node[plane]{}-- (9,-2)node[midway, right]{$\frac{3}{2}$}node[plane]{};
       \draw (9,-2)  node[plane]{} -- (7,-3) node[midway, below]{$6$}node[plane]{}--(5,-2)node[plane]{};
        \draw(5,-2)[dashed]  node[plane]{} -- (5,0)node[midway, left]{$\frac{3\sqrt{3}}{4}$};
        \draw (5,0)node[plane]{}--(9,0)node[plane]{}--(7,-3) node[near end, left]{$6$} node[plane]{};
        \draw (7,-3)node[plane]{}--(7,1)[dashed]node[near start, left]{$2\sqrt{3}$}node[plane]{};
        \draw (7,1)node[plane]{}--(9,-2)node[plane]{}--(5,0)node[plane]{}--(7,1)node[plane]{}--(5,-2)node[near end, right]{$6$}node[plane]{};
		\end{tikzpicture}
	}


  \caption{Two Coxeter cube groups $\Gamma_1$ and $\Gamma_2$ in $\Isom(\mathcal{H}^3)$.}
  \label{idealcubering}
\end{figure}

\end{ex}

\begin{caution}
Two hyperbolic Coxeter groups having the same Vinberg field, the same Vinberg ring and similar Vinberg forms do \textit{not} have to be commensurable. For example, consider the two non-cocompact quasi-arithmetic (but \textit{not} arithmetic) Coxeter cube groups $\Gamma_2$ as above and $\Gamma_3$ in $\Isom(\mathcal{H}^3)$ defined by the graph in Figure \ref{idealcube}. Observe that for both groups, the Vinberg field is $\mathbb{Q}$ and the Vinberg ring is $\mathbb{Z}[1/2]$. Since $\Gamma_2$ and $\Gamma_3$ are quasi-arithmetic, we can apply Maclachlan's criterion to decide whether the Vinberg spaces $(V_2,q_2)$ and $(V_3,q_3)$ are similar (see Theorem \ref{macla2}).

\begin{figure}[h!]	
\centering
\captionsetup[subfloat]{position=bottom,labelformat=empty}
\subfloat[$\Gamma_3$]{
\begin{tikzpicture}[plane/.style={
                      circle, fill=black, minimum size=5pt, inner sep=0pt
                    }, font=\footnotesize
                   ]
  \draw (-3,0) node[plane]{} -- (-1,1)  node[plane]{}--(1,0) node[plane]{};
        \draw[dashed] (1,0)  node[plane]{}--(1,-2)node[midway, right]{$\frac{5}{4}$}node[plane]{};
      \draw (1,-2) node[plane]{} -- (-1,-3) node[plane]{} -- (-3,-2) node[midway, below]{$6$} node[plane]{};
      \draw [dashed](-3,-2) node[plane]{} -- (-3,0)node[midway, left]{$\sqrt{3}$};
      \draw (-3,0) node[plane]{}--(-1,-3)node[plane]{}--(1,0)node[plane]{}--(-3,0)node[plane]{}--(1,-2)node[plane]{}--(-1,1)node[plane]{}--(-3,-2)node[near start, above]{$6$}node[plane]{}--(-1,1) node[plane]{};
      
     \draw [dashed](-1,-3)node[plane]{}--(-1,1)node[near start, left]{$\frac{7}{2}$}node[plane]{};
	\end{tikzpicture}
	}

  \caption{Quasi-arithmetic Coxeter cube group $\Gamma_3$ acting on $\mathcal{H}^3$.}
  \label{idealcube}
\end{figure}
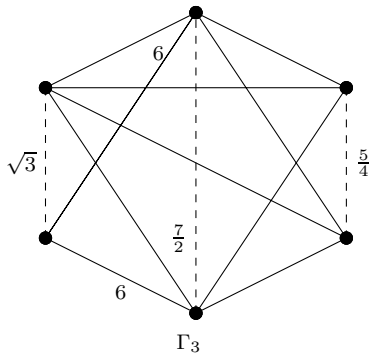

With Vinberg's construction, we can compute the matrices representing $q_1$ and $q_2$ and diagonalise them over $\mathbb{Q}$. We obtain 

\[
q_2=\diag (4,3,15,-15) \text{   and   } q_3=\diag \left(4,3,\frac{-225}{4},\frac{225}{4}\right).
\]

The quadratic forms $q_2$ and $q_3$ have both $(-1,3)$ as Hasse invariant and therefore they have identical Witt invariant represented by the quaternion algebra $B=(1,1)$. Hence, the ramification set of $B\otimes_\mathbb{Q}  \mathbb{Q}(\sqrt{-1})$ over $\mathbb{Q}(\sqrt{-1})$ is identical for both groups. This implies that the Vinberg spaces are similar. However, as shown in \citep{yoshida} by means of a geometric argument, $\Gamma_2$ and $\Gamma_3$ are \textit{not} commensurable.
\end{caution}

\section{New generators for the Vinberg field}\label{chap:five}

In this last section we discuss various aspects of the Vinberg field of a Coxeter group, such as the possible Vinberg fields associated to quasi-arithmetic Coxeter groups and the range of the admissible dihedral angles of a Coxeter polynomial $P$ in terms of the extension degree $d$ of its Vinberg field. We conclude by providing two new sets of generators for the Vinberg field of a \textit{quasi-arithmetic} hyperbolic Coxeter group.

\subsection{The Vinberg field of a quasi-arithmetic hyperbolic Coxeter group}\label{sec:vinfieldQA}

In this part we present some results about possible Vinberg fields associated to quasi-arithmetic Coxeter groups in $ \Isom(\mathcal{H}^n)$, $n\geq 2$. By Remark 1 of \citep{vinberg1967discrete} a non-cocompact quasi-arithmetic Coxeter group has Vinberg field $\mathbb{Q}$. Moreover there are no arithmetic Coxeter groups in $\Isom(\mathcal{H}^n)$ for $n\geq 30$ \citep[Theorem 2.2]{vinberg1993discrete}. 

We start by considering compact hyperbolic Coxeter $n$-simplices, which were classified by Lann\'er and exist only for $n\leq 4$. Their Coxeter graph are called \textit{Lann\'er graph}. For the complete list of Lann\'er graphs see \citep[Table 2.2]{tumarkin2007compact}, for example.

Essential for the following is the fact \citep[Corollary 2.1]{tumarkin2007compact} that the Coxeter graph of a cocompact hyperbolic Coxeter group contains a subgraph which is a Lann\'er graph (called a \textit{Lann\'er subgraph}). 

Consider a cocompact quasi-arithmetic hyperbolic Coxeter group. Its Gram matrix satisfies part $ii)$ of Theorem \ref{thm:vinbergarithmeticity}, which is crucial for the proof of the following result of Vinberg.

\begin{thm}[\citep{vinberg1985absence}, Proposition 17]\label{prop:vinberanddet} For a cocompact quasi-arithmetic hyperbolic Coxeter group, the Vinberg field is generated by the determinant of any Lann\'er subgraph\footnote{The determinant of a Coxeter graph is the determinant of the corresponding Gram matrix.} of the Coxeter graph.
\end{thm}

In \citep[Theorem 2 and Theorem 3]{vinberg1985absence}, Vinberg exploited the above result in order to show that for $n\geq14$ the only possible Vinberg fields of a cocompact arithmetic Coxeter group $\Gamma$ are 
\[
\mathbb{Q}(\sqrt{2}),\mathbb{Q}(\sqrt{3}),\mathbb{Q}(\sqrt{5}),\mathbb{Q}(\sqrt{6}),\mathbb{Q}(\sqrt{2},\sqrt{3}),\mathbb{Q}(\sqrt{2},\sqrt{5}),\mathbb{Q}(\cos2\pi /m)
\]
with $m=7,9,11,15,16,20$. While, for $n\geq 22$, the possible Vinberg fields are
\[
\mathbb{Q}(\sqrt{2}),\mathbb{Q}(\sqrt{5}),\mathbb{Q}(\cos 2\pi/7).
\]

We exploit Theorem \ref{prop:vinberanddet} by assuming that the Coxeter graph contains a Lann\'er subgraph of order at least three and present the following result.

\begin{prop}
Let $\Gamma<\Isom(\mathcal{H}^n)$ be a cocompact quasi-arithmetic Coxeter group such that its Coxeter graph contains a Lann\'er subgraph of order at least three. The possible Vinberg fields for $\Gamma$ are
\[
\mathbb{Q},\mathbb{Q}(\sqrt{2}),\mathbb{Q}(\sqrt{3}), \mathbb{Q}(\sqrt{5}), \mathbb{Q}(\sqrt{6}),\mathbb{Q}(\sqrt{2},\sqrt{3}),\mathbb{Q}(\sqrt{2},\sqrt{5}),\mathbb{Q}(\cos 2\pi/m)
\]
for $m=7,9,11,15,16,20$.
\end{prop}

\begin{proof}
By Theorem \ref{prop:vinberanddet}, the Vinberg field of $\Gamma$ is the extension of $\mathbb{Q}$ by the determinant of any Lann\'er subgraph. Since $\Gamma$ is quasi-arithmetic, any such Lann\'er subgraph of order at least three describes an \textit{arithmetic} Lann\'er group. Indeed, the Gram matrices of the subgroups corresponding to these Lann\'er subgraphs must satisfy part $ii)$ of Theorem \ref{thm:vinbergarithmeticity}. Moreover, Lann\'er graphs of order at least three do not have dotted edges. These two conditions imply that the Lann\'er subgraphs describe arithmetic groups (see \citep[Remark 3]{vinberg1967discrete}). The arithmetic Lann\'er groups of order three have been determined by Takeuchi \citep{takeuchi1977arithmetic}: there are finitely many examples. The Lann\'er groups of orders four and five are all arithmetic with one exception (see \citep{johnson2002commensurability}, for example). Hence, we have to compute finitely many determinants, and the result follows.
\end{proof}

\begin{rmk}
When the Coxeter graph of $\Gamma$ contains only Lann\'er subgraphs of order two, other fields can arise. In his PhD thesis \citep{esselmann1994kompakte}, Esselmann gives examples of such cocompact arithmetic hyperbolic Coxeter groups with Vinberg fields $\mathbb{Q}(\sqrt{13})$, $\mathbb{Q}(\sqrt{17})$ and $\mathbb{Q}(\sqrt{21})$.
\end{rmk}

\subsection{The Vinberg field and the dihedral angles of a hyperbolic Coxeter group}

With the second part of Remark \ref{rmk:trace}, we know that, for $n> 2$, the Vinberg field is a number field. Thus the study of the extension degree of the Vinberg field is an interesting topic. A bound on the extension degree on the Vinberg field for arithmetic hyperbolic Coxeter groups has been studied by many authors. A lot of work has been made by Nikulin \citep{nikulin2011transition}.

In this part we see how the degree of the Vinberg field determines the range of admissible dihedral angles of a Coxeter polyhedron.
\begin{prop}\label{prop:graphwheight}
Let $\Gamma<\Isom(\mathcal{H}^n)$, $n>2$, be a cofinite Coxeter group with Coxeter polyhedron $P$ and Vinberg field $K$ of degree $d$.
Then, for any dihedral angle $\frac{\pi}{m}$ of $P$ one has
\[
\phi(m)\leq 2\, d,
\]
where $\phi(m)$ is the \textit{Euler's totient function}.
\end{prop}
\begin{proof}

Let $a=a_m:=\cos\left(\frac{2\pi}{m}\right)$. It is well-known that (see \citep{lehmer1933note}, for example)
\begin{equation}\label{eq:extensiondegree}
[\mathbb{Q}(a):\mathbb{Q}]=\phi(m)/2,
\end{equation}
where $\phi(m)$ is the Euler's totient function $\phi(m)$, which is the number of positive integers, relatively coprime to $m$, between $1$ and $m$, both included.

Next, consider the Vinberg field $K$ of $\Gamma$. To every dihedral angle $\frac{\pi}{m}$ of $P$ corresponds a subgraph of the Coxeter graph of the form 
\[
\gm
\]
which consists only of two nodes if $m=2$. To this subgraph corresponds the cycle $b_{i_1 i_2}=4\cos^2\left(\frac{\pi}{m}\right)\in K$ (see \eqref{eq:cycles}). This forces $[\mathbb{Q}(4\cos^2\left(\frac{\pi}{m}\right)):\mathbb{Q}]\leq d$. 

By the angle doubling property of the cosine function we have $\mathbb{Q}\left(4\cos^2 \left(\frac{\pi}{m}\right)\right)=\mathbb{Q}\left(\cos \left(\frac{2\pi}{m}\right)\right)$.
Therefore, by \eqref{eq:extensiondegree}, the weight $m$ must satisfy the inequality
\[
[\mathbb{Q}\left(\cos \left(\frac{2\pi}{m}\right)\right):\mathbb{Q}]=\phi(m)/2\leq d.
\]
\end{proof}

\begin{ex}
Let $\Gamma$ be an \textit{arithmetic} Coxeter group in $\Isom(\mathcal{H}^n)$, $n\geq 14$. By \citep[Theorem 2 and Theorem 3]{vinberg1985absence}, the degree $d=[K:\mathbb{Q}]$ is smaller than or equal to five. For $d=5$, Proposition \ref{prop:graphwheight} yields $\phi(m)\leq 10$, for any dihedral angle $\frac{\pi}{m}$. In general, for $x$ not equal to $2$ or $6$, Euler's function $\phi(x)$ satisfies
\[
\phi(x)\geq \sqrt{x}.
\]
Thus, for $ m\leq 100$, we have to check when $\phi(m)\leq 10$. As a result, for $n\geq 14$, all the possible values for $m$ are 
\[
 2,3,4,5,6,7,8,9,10,11,12,14,15,16,18,20,22,24,30 .
\]
\end{ex}

\subsection{The Gram field of a hyperbolic Coxeter group}

For $n\geq 2$, consider a Coxeter group $\Gamma$ in $\Isom(\mathcal{H}^n)$ of rank $N$. Let $G$ be its Gram matrix of signature $(n,1)$ with characteristic polynomial 
\[
\chi_G(t)=a_0+a_1t+\dots +a_{N}t^{N},\quad a_{N}=1.
\]
The matrix $G$ is uniquely defined by $\Gamma$ up to simultaneous permutation of its lines and columns which would yield a similar matrix $G'$ with identical characteristic polynomial.

Notice that $a_{N-1}=(-1)^{N-1}\Tr(G)=(-1)^{N-1}N$. Moreover, each coefficient $a_r$ of $\chi_G$, $r< N$, can be expressed as the sum of all the principal minors of size $N-r$ (see \citep[p. 53]{horn2012matrix}, for example). In particular, $a_r$ vanishes for all $r<N-(n+1)$.

\begin{defn}
Let $\Gamma$ be a hyperbolic Coxeter group of rank $N$. Let $G$ be its Gram matrix with characteristic polynomial $\chi_G(t)=a_0+a_1t+\dots +a_{N}t^{N},\quad a_{N}=1$. The \textit{Gram field} $K(G)$ is the field generated by the coefficients of $\chi_G(t)$ over $\mathbb{Q}$, namely 
\[
K(G)=\mathbb{Q}(a_j\mid 0\leq j\leq N).
\]
\end{defn}

\begin{prop}\label{prop:fieldsequality}
Let $\Gamma$ be a cofinite quasi-arithmetic hyperbolic Coxeter group with Vinberg field $K$. Then
\[
K=K(G).
\]
\end{prop}

\begin{proof}
We prove first the inclusion $K\supseteq K(G)$. By \citep[Proposition 11]{vinberg1985absence}, the determinant of the Gram matrix $G$ is given by a sum of cyclic products. The same result applies to every principal submatrix of $G$. Since the coefficients of $\chi_G$ can be expressed as the sum of principal minors of $G$ (see \citep[p. 53]{horn2012matrix}, for example), we get $K\supseteq K(G)$.

Assume that $K\supsetneq K(G)$. Then there exists a non-trivial embedding $\sigma:K\hookrightarrow \mathbb{R}$ which is the identity on $K(G)$. Let $G^\sigma$ be the matrix obtained by applying $\sigma$ to every coefficient of $G$ and let $\chi_G=\sum_{i=0}^Na_ix^i$ be the characteristic polynomial of $G$. Since $\sigma$ is a field homomorphism, then $\chi_{G^\sigma}=\sum_{i=0}^N \sigma(a_i)x^i$. The embedding $\sigma$ fixes the coefficients of $\chi_G$, thus 
\[
\chi_G=\chi_{G^{\sigma}}.
\]
In particular, $G^\sigma$ has signature $(n,1)$ and is not positive semidefinite. This is a contradiction to part $ii)$ of Theorem \ref{thm:vinbergarithmeticity} and the claim follows.
\end{proof}

\subsection{The Coxeter field of a hyperbolic Coxeter group}

Let $\Gamma<\Isom(\mathcal{H}^n)$, $n\geq 2$, be a cofinite Coxeter group with natural set of generators $\{s_1,\dots ,s_N\}$. Consider a Coxeter transformation $C=s_1\cdots s_N$ of $\Gamma$ defined up to the ordering of the factors. With the real coefficients of the characteristic polynomial $\chi_C(t)$ we define a new field, the \textit{Coxeter field}, and prove that it coincides with the Vinberg field $K(\Gamma)$ if $\Gamma$ is quasi-arithmetic.

The proof is based on the work of Howlett \citep{howlett1982coxeter} and the theory of $M$-matrices which we are going to review briefly.

Let $W=(W,S)$ be a Coxeter system with generating set $S=\{s_1,\dots ,s_N\}$ satisfying the relations of a Coxeter group. By Tits' theory, it is known that $W$ can be represented as a subgroup of $\GL(V)$ for a real vector space $V$ of dimension $N$ equipped with a suitable symmetric bilinear form $B$ (see \citep{humphreys1992reflection}, for example). Denote by $\radic(V)=\{v\in V\mid B( v, v')=0\quad\forall v'\in V\}$ the \textit{radical} of $B$ which will play a role later on. A \textit{Coxeter element} $c\in W$ is the product of the $N$ generators in $S$ arranged in any order. The representative $C_T\in\GL(V)$ of $c$ is called a \textit{Coxeter transformation} of $W$. 

For a Coxeter element $c=s_1\cdots s_N$, the matrix of $C_T$ with respect to a basis $\{v_1,\dots ,v_N\}$ of $V$, denoted again by $C_T$, can be written according to (see \citep{howlett1982coxeter}, for example)
\begin{equation}\label{eq:CU}
C_T=-U^{-1}U^{T},
\end{equation}
where $U\in\GL(N,\mathbb{R})$  is given by
\begin{equation}\label{eq:definitionU}
U=
    \left(
    \begin{array}{ccccc}
    1                                    \\
      & 1             &   & \ast\\
      &               & \ddots                \\
      & \text{\huge0} &   & 1            \\
      &               &   &   & 1
    \end{array}
    \right),
\end{equation}
with $[U]_{st}=2 B( v_s, v_t)$ for $t>s$. Notice that 
\begin{equation}\label{eq:2B}
U+U^T=2B.
\end{equation}

By means of the theory of $M$-matrices, Howlett (\citep[Theorem 4.1]{howlett1982coxeter}, see also \citep{a1976valeurs}) characterised abstract Coxeter groups in terms of a Coxeter transformation $C_T$ and its eigenvalues. More concretely, an \textit{M-matrix} is a real matrix with non-positive off-diagonal entries all of whose principal minors are positive. For example, the matrix $U$ given by \eqref{eq:definitionU} is an $M$-matrix. 

The proof of Howlett's Theorem 4.1 in \citep{howlett1982coxeter} is based on the following results.

\begin{lemma}[\citep{howlett1982coxeter}, Lemma 3.1]\label{prop:howlett1} 
Let $U$ be a real matrix such that $U+U^T$ is positive definite. Then $U$ is invertible and $-U^{-1}U^T$ is diagonalisable over $\mathbb{C}$ with all of its eigenvalues having modulus one.
\end{lemma}

\begin{lemma}[\citep{howlett1982coxeter}, Lemma 3.2 and Corollary 3.3]\label{prop:howlett2} Let $U$ be an $M$-matrix such that $U+U^T$ is not positive definite. Then $-U^{-1}U^T$ has a real eigenvalue $\lambda\geq 1$. If $U+U^T$ is not positive semidefinite, then $\lambda>1$. If $U+U^T$ is positive semidefinite, all the eigenvalues of $-U^{-1}U^T$ have modulus one and $-U^{-1}U^T$ is not diagonalisable.
\end{lemma}

Later we will also need another lemma, which is stated in Howlett's proof of Lemma \ref{prop:howlett2}.

\begin{lemma}\label{prop:howlett3}
Let $U$ be an invertible real matrix such that $U+U^T$ is positive semidefinite. Then the eigenvalues of $-U^{-1}U^T$ have all modulus one.
\end{lemma}

\begin{proof}
For $\epsilon>0$ define the matrix $U^\epsilon:=U+\epsilon I$. Since $U+U^T$ is positive semidefinite, $U^\epsilon+(U^\epsilon)^T$ is positive definite. By Lemma \ref{prop:howlett1}, all the eigenvalues of $-(U^\epsilon)^{-1}(U^\epsilon)^T$ have modulus one. The entries of $U^\epsilon$ depend continuously on $\epsilon$. The same can be said for $-(U^\epsilon)^{-1}(U^\epsilon)^T$ and the coefficients of its characteristic polynomial. Hence the eigenvalues of $-(U^\epsilon)^{-1}(U^\epsilon)^T$ and their modulus depend continuously on $\epsilon$, and the claim follows.
\end{proof}

Let $\Gamma<\Isom(\mathcal{H}^n)$ be a hyperbolic Coxeter group with generating reflections $s_1, \dots , s_N$. In this way $\Gamma$ represents a geometric realisation of an abstract Coxeter group. Let $P\in\mathcal{H}^n$ be its Coxeter polyhedron with outer unit normal vectors $e_1,\dots , e_N$ and associated Gram matrix $G\in\Mat(N,\mathbb{R})$.

Let $C\in\Gamma$ be a Coxeter transformation of $\Gamma$. Our goal is to construct a new field $K(C)$ associated to $C$ which we can identify later with the Vinberg field $K(\Gamma)$. Our motivation comes from \citep[Theorem 1.8, (iv)]{reiner2017non}, due to Reiner, Ripoll and Stump, relating Coxeter transformations of a finite complex reflection group to its field of definition\footnote{In \citep{reiner2017non}, the field of definition of a Coxeter group $W$ is the field generated by all the traces of the matrices representing the elements of $W$.} (see also Malle in \citep[Section 7A]{malle19992}). 

Inspired by this, we state the following definition. 

\begin{defn}
Let $\Gamma$ be a hyperbolic Coxeter group. Let $C\in\Gamma$ be a Coxeter transformation with characteristic polynomial $\chi_C(t)=a_0+a_1t+\dots +a_{n+1}t^{n+1}$, $a_{n+1}=1$. The \textit{Coxeter field} $K(C)$ is the field generated by the coefficients of $\chi_C(t)$ over $\mathbb{Q}$, namely 
\[
K(C)=\mathbb{Q}(a_j\mid 0\leq j\leq n+1).
\]
\end{defn}

It is not difficult to see that $\chi_C(t)$ is palindromic ($a_j=a_{n+1-j}$) if $N=n+1+2k$ and it is pseudo-palindromic ($a_j=-a_{n+1-j}$) if $N=n+1+(2k+1)$, for some $k\geq 0$.

Furthermore, $N-(n+1)$ is the dimension of the radical $\radic \left(\mathbb{R}^N\right)$ for the Tits representation space $\left(\mathbb{R}^N,G\right)$. Clearly, every element in $\Gamma$ viewed in $\GL\left(\mathbb{R}^N\right)$ acts as the identity on $\radic\left(\mathbb{R}^N\right)$. Hence the same is true for every Coxeter transformation $C_T\in\GL\left(\mathbb{R}^N\right)$ of $\Gamma$. Since $\dim\left(\mathbb{R}^N/\radic\left(\mathbb{R}^N\right)\right)=n+1$, the characteristic polynomials $\chi_C$ and $\chi_{C_T}$ are related by
\begin{equation}\label{eq:charpoly}
(t-1)^{(N-(n+1))}\chi_C(t)=\chi_{C_T}(t).
\end{equation}

In particular the field generated by the coefficients of $\chi_{C}$ and the field generated by the coefficients of $\chi_{C_T}$ coincide. With this preparation we are ready to prove the following result.
\begin{prop}\label{prop:QAconj}
Let $\Gamma$ be a cofinite quasi-arithmetic hyperbolic Coxeter group with Vinberg field $K$, and let $C$ be any Coxeter transformation of $\Gamma$. Then
\[
K=K(C).
\]
\end{prop}
\begin{proof}
We first show that $K\supseteq K(C)$. The Vinberg field $K$ is a field of definition (see Section \ref{Vinconstruction}). Thus, by means of a suitable basis, the Coxeter transformation $C$ can be written as a matrix with coefficients in $K$. Since the characteristic polynomial is invariant under a basis change, we have that $K\supseteq K(C)$.

Assume that $K\supsetneq K(C)$. Then there exists a non-trivial embedding $\sigma:K\hookrightarrow \mathbb{R}$ that is the identity on $K(C)$.  Consider the Coxeter transformation $C_T$ acting on $\left(\mathbb{R}^N,G\right)$ which corresponds to $C$ in the sense of Tits. By \eqref{eq:CU}, we can express $C_T=-U^{-1}U^T$ where $U$ is an $M$-matrix. As a consequence of Lemma \ref{prop:howlett2}, $C_T$ has a real eigenvalue $\lambda>1$. 

Consider the matrices $G^\sigma$ and $U^\sigma$, obtained by applying $\sigma$ to the coefficients of the Gram matrix $G$ of $\Gamma$ and $U$. The matrix $U^\sigma$ is invertible but in general not an $M$-matrix anymore since its off-diagonal entries may become positive. Define $C^\sigma:=-(U^\sigma)^{-1}(U^\sigma)^T$. By \eqref{eq:2B}, we have the equation $U^\sigma+(U^\sigma)^T=2\,G^\sigma$. By part $ii)$ of Theorem \ref{thm:vinbergarithmeticity}, $U^\sigma+(U^\sigma)^T$ is therefore positive semidefinite. The embedding $\sigma$ is a field homomorphism, thus the characteristic polynomial of $C^\sigma$ is obtained by applying $\sigma$ to the coefficients of the characteristic polynomial of $C_T$. Since $\sigma$ is the identity on $K(C)$, it leaves the characteristic polynomial $\chi_C$ invariant. The identity \eqref{eq:charpoly} then yields 
\[
\chi_{C_T}=\chi_{C^\sigma}.
\] 

This is contradiction, since $C_T$ has an eigenvalue $\lambda>1$ and since, by Lemma \ref{prop:howlett3}, all the eigenvalues of $C^\sigma$ have modulus one. 
\end{proof}

\bibliographystyle{abbrv}
\bibliography{Dotti-CHCG}{}

\end{document}